\documentclass[12pt]{article}
\usepackage{graphicx}
\usepackage[ruled,vlined]{algorithm2e}
\usepackage{appendix}
\usepackage{amsmath}
\usepackage{amssymb}
\usepackage{amsthm}
\usepackage{multirow}
\usepackage{color}
\usepackage{longtable}
\usepackage{array}
\usepackage{url}
\usepackage{booktabs}
\usepackage{float}
\usepackage{mathtools}
\usepackage{tikz}
\usepackage{caption}

\newtheorem{theorem}{Theorem}[section]
\newtheorem{definition}[theorem]{Definition}
\newtheorem{lemma}[theorem]{Lemma}

\newtheorem{construction}[theorem]{Construction}
\newtheorem{corollary}[theorem]{Corollary}

\newtheorem{problem}[theorem]{Problem}
\newtheorem{conjecture}[theorem]{Conjecture}

\theoremstyle{definition}
\newtheorem{remark}[theorem]{Remark}

\newcommand{\turan}{Tur\'{a}n }

\title{On the Tur\'{a}n number of 1-subdivision of $K_{3,t}$}
\author{Tao Zhang$^{\text{a,}}$\thanks{Research supported by the National Natural Science Foundation of China under Grant No. 11801109.},~ Zixiang Xu$^{\text{b}}$~ and  Gennian Ge$^{\text{b,}}$\thanks{Corresponding author (e-mail: gnge@zju.edu.cn). Research supported by the National Natural Science Foundation of China under Grant Nos. 11431003, 61571310 and 11971325, Beijing Scholars
Program, Beijing Hundreds of Leading Talents Training Project of Science and Technology, and Beijing Municipal Natural Science Foundation.}\\
\footnotesize $^{\text{a}}$ School of Mathematics and Information Science, Guangzhou University, Guangzhou 510006, China.\\
\footnotesize $^{\text{b}}$ School of Mathematics Sciences, Capital Normal University, Beijing 100048, China.\\
}

\begin{document}

\date{}

\maketitle

\begin{abstract}
For a graph $H$, the 1-subdivision of $H$, denoted by $H'$, is the graph obtained by replacing the edges of $H$ by internally disjoint paths of length 2. Recently, Conlon, Janzer and Lee (arXiv: 1903.10631) asked the following question: For any integer $s\ge2$, estimate the smallest $t$ such that $\textup{ex}(n,K_{s,t}')=\Omega(n^{\frac{3}{2}-\frac{1}{2s}})$. In this paper, we consider the case $s=3$. More precisely, we provide an explicit construction giving
\begin{align*}
\text{ex}(n,K_{3,30}')=\Omega(n^{\frac{4}{3}}),
\end{align*}
which reduces the estimation for the smallest value of $t$ from a magnitude of $10^{56}$ to the number $30$. The construction is algebraic, which is based on some equations over finite fields.
\medskip


\smallskip

\end{abstract}

\section{Introduction}
Given a graph $H$, the Tur\'{a}n number $\text{ex}(n,H)$ is the maximum number of edges in an $n$-vertex graph that does not contain $H$ as a subgraph. The estimation of $\text{ex}(n,H)$ for various graphs $H$ is one of the most important problems in extremal graph theory. For general graph $H,$ the famous Erd\H{o}s-Stone-Simonovits Theorem \cite{1946ErodsBAMS} gives that
\begin{align*}
\text{ex}(n,H)=(1-\frac{1}{\chi(H)-1}+o(1))\binom{n}{2},
\end{align*}
where $\chi(H)$ is the chromatic number of $H$. This theorem asymptotically solves the problem when $\chi(H)\ge3$. However, for bipartite graphs $H$, it only gives $\text{ex}(n,H)=o(n^{2})$. So far, there are only a few bipartite graphs for which the asymptotics of whose Tur\'{a}n numbers are known. A well-known theorem of K\"{o}vari, S\'{o}s and \turan \cite{Kovari1954} showed that $\text{ex}(n,K_{s,t})=O(n^{2-1/s})$ for any integers $t\ge s$. When $s=2,3,$ matched lower bounds were found in \cite{Brown1966, Erdos1966}. For general values of $s$, we have $\text{ex}(n,K_{s,t})=\Omega(n^{2-\frac{1}{s}})$ when $t\geqslant (s-1)!+1$ \cite{ARS99, KRS96}. Recently, using the random algebraic method, Bukh \cite{Bukh2015RAC} gave a new construction of $K_{s,t}$-free graphs which also yields $\text{ex}(n,K_{s,t})=\Omega(n^{2-\frac{1}{s}}),$ where $t$ is sufficiently large.

The $1$-subdivision of a graph $H$, denoted by $H'$, is the graph obtained by replacing the edges of $H$ by internally disjoint paths of length $2$.
 Recently, the $1$-subdivision of graphs received a lot of attention due to a paper of Kang, Kim and Liu \cite{KKL18}, where they made the following conjecture about the $1$-subdivision of a general bipartite graph.
\begin{conjecture}[\cite{KKL18}]\label{conj:subdivision}
Let $H$ be a bipartite graph. If $\textup{ex}(n,H)=O(n^{1+\alpha})$ for some $\alpha>0$, then $\textup{ex}(n,H')=O(n^{1+\frac{\alpha}{2}})$.
\end{conjecture}
Apart from being interesting on its own, somewhat surprisingly, Kang, Kim and Liu \cite{KKL18} showed that this seemingly unrelated conjecture implies the rational exponent conjecture as follows.
\begin{conjecture}[\cite{1988Erdos}]\label{conj:rational}
  For every rational number $r\in [1,2],$ there exists a graph $F$ with $\textup{ex}(n,F)=\Theta(n^{r}).$
\end{conjecture}
For more information on the recent active study of the Tur\'{a}n problem for subdivisions, we refer the readers to \cite{CJL19, CL18, 2019EJCJanzer, Janzer19, JML2018, JQ19} and the references therein.

In this paper, we focus on the 1-subdivision of complete bipartite graphs. In \cite{CJL19}, Conlon, Janzer and Lee showed that $\text{ex}(n,K_{s,t}')=O(n^{\frac{3}{2}-\frac{1}{2s}})$ for $2\le s\le t$, which proved Conjecture~\ref{conj:subdivision} for complete bipartite graphs. Combining the random algebraic construction in \cite{Bukh2018}, they also showed the upper bound is tight when $t$ is sufficiently large compared to $s$. Since the random algebraic method requires the parameter $t$ to be very large, in the same paper, they asked the following problem:
\begin{problem}
For any integer $s\ge2$, estimate the smallest $t$ such that $\textup{ex}(n,K_{s,t}')=\Omega(n^{\frac{3}{2}-\frac{1}{2s}})$.
\end{problem}
The previously known smallest value of $t$ from random algebraic method in \cite{Bukh2018} is $s^{O(s^{2})}.$ In particular, when $s=3,$ their construction showed that $t\approx 10^{56}.$ Moreover, the random algebraic method falls well short of this due to Lang-Weil bound \cite{LangWeil1954}. The case $s=2$ amounts to estimating the extremal number of the theta graph $\theta_{4,t}$. Very recently, Verstra\"{e}te and Williford \cite{VW19} gave an algebraic construction which yields $\text{ex}(n,\theta_{4,3})=\Omega(n^{\frac{5}{4}})$.  However deriving a similar bound for $\text{ex}(n, \theta_{4,2})$ is likely to be very difficult, as it would solve the famous open problem of estimating $\text{ex}(n, C_{8})$. In this paper, we consider the next case $s=3$, and prove the following result.
\begin{theorem}\label{thm}
$\textup{ex}(n,K_{3,30}')=\Omega(n^{\frac{4}{3}})$.
\end{theorem}
Combining with the above upper bound, we have
\begin{corollary}
$\textup{ex}(n,K_{3,30}')=\Theta(n^{\frac{4}{3}})$.
\end{corollary}

The rest of this paper is organized as follows. In Section~\ref{pre}, we will give some basics about the resultant of polynomials. In Section~\ref{mainpart}, we prove our main result. Section~\ref{conclusion} concludes our paper. All computations have been done by MAGMA \cite{BCP97}.
\section{Preliminaries}\label{pre}
Our main technique is the resultant of polynomials, which has been used in \cite{ZXG19}. For the convenience of readers, we recall some basics about the resultant of polynomials, which will be used in the following section. Let $\mathbb{F}$ be a field, and $\mathbb{F}[x]$ be the polynomial ring with coefficients in $\mathbb{F}$.
\begin{definition}
Let $f(x),g(x)\in\mathbb{F}[x]$ with $f(x)=a_{m}x^{m}+\cdots+a_{1}x+a_{0}$ and $g(x)=b_{n}x^{n}+\cdots+b_{1}x+b_{0}$, then the resultant of $f$ and $g$ is defined by the determinant of the following $(m+n+2)\times (m+n+2)$ matrix,
\begin{align*}
\left(
            \begin{array}{cccccccc}
              a_{0} & a_{1} & \cdots & a_{m} &   &   &&   \\
               & a_{0} & \cdots  & a_{m-1} & a_{m} && & \\
                &   & \cdots  & \cdots  & \cdots & &  &\\
              &  &   &   &   & a_{0} & \cdots  & a_{m}\\
              b_{0} & b_{1} & \cdots &\cdots & b_{n} &     &\\
               & b_{0} & \cdots & \cdots  & b_{n-1} & b_{n} & & \\
                &   &  \cdots &  \cdots &\cdots  & \cdots  &   &\\
                &   &   &   &b_{0} & \cdots& \cdots  & b_{n}\\
            \end{array}
          \right),
\end{align*}
which is denoted by $R(f,g)$.
\end{definition}
The resultant of two polynomials has the following property.
\begin{lemma}[\cite{Fuhrmann2012}]
If $\text{gcd}(f(x),g(x))=h(x)$, where $\text{deg}(h(x))\ge1$, then $R(f,g)=0$. In particular, if $f$ and $g$ have a common root in $\mathbb{F}$, then $R(f,g)=0$.
\end{lemma}

When we consider multivariable polynomials, we can define the resultant similarly, and the above lemma still holds when we fix one variable. For any $f,g\in\mathbb{F}[x_{1},\dots,x_{n}]$, let $R(f,g,x_{i})$ denote the resultant of $f$ and $g$ with respect to $x_{i}$, then we have $R(f,g,x_{i})\in\mathbb{F}[x_{1},\dots,x_{i-1},x_{i+1},x_{n}]$.
\section{Construction of $K_{3,30}'$-free graphs}\label{mainpart}
In this section, we construct a $K_{3,30}'$-free graph with $n$ vertices and $\Omega(n^{\frac{4}{3}})$ edges. Let $\mathbb{F}_{p}$ be a finite field, where $p$ is an odd prime with $p\equiv5\pmod{6}$ and $p>11$. Let $S=\{x: x \in \mathbb{F}_{p}, x\in [1,\frac{p-5}{6}]\}$. Then we have the following lemma.
\begin{lemma}\label{lemma3}
For any $x,y,z,t\in S$, we have $x+y\ne0$, $x+y+z+t\ne0$, $x+5y\ne0$, and $x^{2}+xy+y^{2}\ne0$.
\end{lemma}
\begin{proof}
Since $x,y,z,t\in \mathbb{F}_{p}$ and $x,y,z,t\in[1,\frac{p-5}{6}]$, then we have $x+y\ne0$, $x+y+z+t\ne0$, $x+5y\ne0$.

Note that $p\equiv5\pmod6$, then $-3$ is a non-quadratic residue module $p$. Hence, $x^{2}+xy+y^{2}=(x+\frac{y}{2})^{2}+3(\frac{y}{2})^{2}\ne0$.
\end{proof}

Now we can give our construction.
\begin{construction}
 The graph $G_{p}$ is defined with vertex set $V:=S\times\mathbb{F}_{p}\times\mathbb{F}_{p}$, where $x=(x_{1},x_{2},x_{3})\in V$ is joined to $y=(y_{1},y_{2},y_{3})\in V$ if $x\ne y$ and
\begin{align*}
&x_{2}+y_{3}=x_{1}y_{1}^{2},\\
&x_{3}+y_{2}=x_{1}^{2}y_{1}.
\end{align*}
\end{construction}
By fixing a vertex $x$, it is easy to see that the choice of $y_{1}$ determines a unique neighbor $y$ of $x$ except $x=y$. Therefore each vertex has degree at least $\frac{p-11}{6}$. Hence $G_{p}$ has $n:=\frac{(p-5)p^{2}}{6}$ vertices and at least $\frac{1}{72}(p-5)(p-11)p^{2}=\Omega(n^{\frac{4}{3}})$ edges. In the following of this section, we will prove that $G_{p}$ is $K_{3,30}'$-free.

We begin with the following simple lemma.
\begin{lemma}\label{lemma1}
If $x,y\in V$ are distinct and have a common neighbor, then $x_{1}\ne y_{1}$, $x_{2}\ne y_{2}$ and $x_{3}\ne y_{3}$.
\end{lemma}
\begin{proof}
We will only prove $x_{3}\ne y_{3}$, the others are similar. Suppose $x,y$ have a common neighbor $u$, then
\begin{align*}
&x_{3}+u_{2}=x_{1}^{2}u_{1},\\
&y_{3}+u_{2}=y_{1}^{2}u_{1}.
\end{align*}
If $x_{3}=y_{3}$, then $u_{1}(x_{1}^{2}-y_{1}^{2})=0$. Hence $x_{1}=y_{1}$ or $x_{1}+y_{1}=0$. If $x_{1}=y_{1}$, then it is easy to get that $x=y$, which is a contradiction. If $x_{1}+y_{1}=0$, this contradicts to the definition of $S$. Hence $x_{3}\ne y_{3}$.
\end{proof}
For any given $a,b,c\in V$ with $a,b,c$ pairwise distinct, we estimate the number of sequences $(x,y,z, w)\in V^{4}$ with $x,y,z,w$ pairwise distinct, such that $ax,xw,by,yw,cz,zw$ are edges in $G_{p}$. We will prove that there are at most $29$ different such sequences. By the definition of graph $G_{p}$, we have
\begin{align*}
&a_{2}+x_{3}=a_{1}x_{1}^{2}, && a_{3}+x_{2}=a_{1}^{2}x_{1},\\
&w_{2}+x_{3}=w_{1}x_{1}^{2}, && w_{3}+x_{2}=w_{1}^{2}x_{1},\\
&b_{2}+y_{3}=b_{1}y_{1}^{2}, && b_{3}+y_{2}=b_{1}^{2}y_{1},\\
&w_{2}+y_{3}=w_{1}y_{1}^{2}, && w_{3}+y_{2}=w_{1}^{2}y_{1},\\
&c_{2}+z_{3}=c_{1}z_{1}^{2}, && c_{3}+z_{2}=c_{1}^{2}z_{1},\\
&w_{2}+z_{3}=w_{1}z_{1}^{2}, && w_{3}+z_{2}=w_{1}^{2}z_{1}.
\end{align*}
Cancelling $x_{2},x_{3},y_{2},y_{3},z_{2},z_{3}$ from the above equations, we can get the following equations
\begin{align}
&f_{1}:=a_{2}-w_{2}-x_{1}^{2}(a_{1}-w_{1})=0,\label{eq1}\\
&f_{2}:=a_{3}-w_{3}-x_{1}(a_{1}^{2}-w_{1}^{2})=0,\label{eq2}\\
&f_{3}:=b_{2}-w_{2}-y_{1}^{2}(b_{1}-w_{1})=0,\label{eq3}\\
&f_{4}:=b_{3}-w_{3}-y_{1}(b_{1}^{2}-w_{1}^{2})=0,\label{eq4}\\
&f_{5}:=c_{2}-w_{2}-z_{1}^{2}(c_{1}-w_{1})=0,\label{eq5}\\
&f_{6}:=c_{3}-w_{3}-z_{1}(c_{1}^{2}-w_{1}^{2})=0.\label{eq6}
\end{align}
In the following of this section, we divide our discussions into three subsections.
\subsection{$a_{1}=b_{1}$ or $a_{1}=c_{1}$ or $b_{1}=c_{1}$}
Without loss of generality, we assume that $a_{1}=b_{1}$. Then Equations (\ref{eq1})-(\ref{eq6}) become
\begin{align*}
&f_{1}:=a_{2}-w_{2}-x_{1}^{2}(a_{1}-w_{1})=0,\\
&f_{2}:=a_{3}-w_{3}-x_{1}(a_{1}^{2}-w_{1}^{2})=0,\\
&f_{3}:=b_{2}-w_{2}-y_{1}^{2}(a_{1}-w_{1})=0,\\
&f_{4}:=b_{3}-w_{3}-y_{1}(a_{1}^{2}-w_{1}^{2})=0,\\
&f_{5}:=c_{2}-w_{2}-z_{1}^{2}(c_{1}-w_{1})=0,\\
&f_{6}:=c_{3}-w_{3}-z_{1}(c_{1}^{2}-w_{1}^{2})=0.
\end{align*}
We begin with the following lemma.
\begin{lemma}\label{lemma4}
If $a_{1}=b_{1}$, then $a_{2}\ne b_{2}$ and $a_{3}\ne b_{3}$.
\end{lemma}
\begin{proof}
If $a_{2}=b_{2}$, then by the above equations, we have $x_{1}^{2}=y_{1}^{2}$, hence $x_{1}=y_{1}$ or $x_{1}=-y_{1}$, which contradicts to Lemmas~\ref{lemma3} and \ref{lemma1}.

If $a_{3}=b_{3}$, then by the above equations again, we have $x_{1}=y_{1}$, which contradicts to Lemma~\ref{lemma3}.
\end{proof}
Now we regard $f_{i}$ ($i=1,2,\dots,6$) as polynomials with variables $x_{1},y_{1},z_{1},w_{1},w_{2},w_{3}$. By a MAGMA program, we can get that
\begin{align*}
&R(f_{1},f_{2},x_{1})=g_{1}\cdot(a_{1}-w_{1}),\\
&R(f_{3},f_{4},y_{1})=g_{2}\cdot(a_{1}-w_{1}),\\
&R(f_{5},f_{6},z_{1})=g_{3}\cdot(c_{1}-w_{1}).
\end{align*}
By Lemma~\ref{lemma1}, we have $a_{1}\ne w_{1}$ and $c_{1}\ne w_{1}$. Then we can compute to get that
\begin{align*}
&R(g_{1},g_{2},w_{2})=g_{4}\cdot(a_{1}-w_{1})(a_{1}+w_{1})^{2},\\
&R(g_{1},g_{3},w_{2})=g_{5}.
\end{align*}
By Lemmas~\ref{lemma3} and \ref{lemma1}, $a_{1}+w_{1}\ne0$ and $a_{1}-w_{1}\ne0$. Let $h=R(g_{4},g_{5},w_{3})$, then $h$ is a polynomial of $w_{1}$ with degree $8$. We can write $h$ as $h=\sum_{i=0}^{8}h_{i}w_{1}^{i}$. Then we can compute to get that
\begin{align*}
h_{8}=(a_{2}-b_{2})^{2}(a_{1}-c_{1}).
\end{align*}
By Lemma~\ref{lemma4}, $a_{2}\ne b_{2}$ and $a_{3}\ne b_{3}$.
If $a_{1}\ne c_{1}$, then there are at most 8 solutions for $w_{1}$. For any fixed $w_{1}$, $g_{4}$ is a polynomial of $w_{3}$ with degree 1. We write $g_{4}$ as $g_{4}=s_{1}w_{3}+s_{0}$, then $s_{1}=a_{3}-b_{3}\ne0$. Hence there is at most $1$ solution for $w_{3}$. If $w_{1}$ and $w_{3}$ are given, then all the remaining variables are uniquely determined. Hence there are at most $8$ different sequences of $(x,y,z,w)$ for this case.

If $a_{1}=c_{1}$, then we can compute to get that $g_{5}=g_{5}'\cdot(a_{1}-w_{1})(a_{1}+w_{1})^{2}$. Let $h'=R(g_{4},g_{5}',w_{3})$, then $h'$ is a polynomial of $w_{1}$ with degree $3$.  We can write $h'$ as $h'=\sum_{i=0}^{3}k_{i}w_{1}^{i}$. Regarding $k_{i}$ ($i=0,1,2,3$) as polynomials with variable $a_{2}$, then by a MAGMA program, we have
\begin{align*}
R(k_{0},k_{3},a_{2})=(a_{3}-b_{3})^{2}(a_{3}-c_{3})(a_{3}-b_{3}).
\end{align*}
By Lemma~\ref{lemma4}, $R(k_{0},k_{3},a_{2})\ne0$. Hence there are at most 3 solutions for $w_{1}$. Similarly as above, for any fixed $w_{1}$, there is at most $1$ solution for $w_{3}$. If $w_{1}$ and $w_{3}$ are given, then all the remaining variables are uniquely determined. Hence there are at most $3$ different sequences of $(x,y,z,w)$ for this case.

Therefore, if $a_{1}=b_{1}$ or $a_{1}=c_{1}$ or $b_{1}=c_{1}$, then there are at most $8$ different sequences of $(x,y,z,w)$.
\subsection{$a_{1}\ne b_{1}$, $a_{1}\ne c_{1}$, $b_{1}\ne c_{1}$ and $a_{2}=b_{2}$ or $a_{2}=c_{2}$ or $b_{2}=c_{2}$}
Without loss of generality, we assume that $a_{2}=b_{2}$. Then Equations (\ref{eq1})-(\ref{eq6}) become
\begin{align*}
&f_{1}:=a_{2}-w_{2}-x_{1}^{2}(a_{1}-w_{1})=0,\\
&f_{2}:=a_{3}-w_{3}-x_{1}(a_{1}^{2}-w_{1}^{2})=0,\\
&f_{3}:=a_{2}-w_{2}-y_{1}^{2}(b_{1}-w_{1})=0,\\
&f_{4}:=b_{3}-w_{3}-y_{1}(b_{1}^{2}-w_{1}^{2})=0,\\
&f_{5}:=c_{2}-w_{2}-z_{1}^{2}(c_{1}-w_{1})=0,\\
&f_{6}:=c_{3}-w_{3}-z_{1}(c_{1}^{2}-w_{1}^{2})=0.
\end{align*}
Now we regard $f_{i}$ ($i=1,2,\dots,6$) as polynomials with variables $x_{1},y_{1},z_{1},w_{1},w_{2},w_{3}$. By a MAGMA program, we can get that
\begin{align*}
&R(f_{1},f_{2},x_{1})=g_{1}\cdot(a_{1}-w_{1}),\\
&R(f_{3},f_{4},y_{1})=g_{2}\cdot(b_{1}-w_{1}),\\
&R(f_{5},f_{6},z_{1})=g_{3}\cdot(c_{1}-w_{1}).
\end{align*}
By Lemma~\ref{lemma1}, we have $a_{1}\ne w_{1}, b_{1}\ne w_{1}$ and $c_{1}\ne w_{1}$. Then we can compute to get that
\begin{align*}
&R(g_{1},g_{2},w_{2})=g_{4},\\
&R(g_{1},g_{3},w_{2})=g_{5},\\
&R(g_{4},g_{5},w_{3})=h\cdot (a_{1}-w_{1})^{2}(a_{1}+w_{1})^{4},
\end{align*}
where $h$ is a polynomial of $w_{1}$ with degree $10$. By Lemmas~\ref{lemma3} and \ref{lemma1}, $a_{1}\ne w_{1}$ and $a_{1}+w_{1}\ne0$. We can write $h$ as $h=\sum_{i=0}^{10}h_{i}w_{1}^{i}$. Then we can compute to get that
\begin{align*}
h_{10}=(a_{2}-c_{2})^{2}(a_{1}-b_{1})^{2}.
\end{align*}

If $a_{2}\ne c_{2}$, then there are at most 10 solutions for $w_{1}$. For any fixed $w_{1}$, $g_{4}$ and $g_{5}$ are polynomials of $w_{3}$ with degree 2. We write $g_{4}$ and $g_{5}$ as $g_{4}=\sum_{i=0}^{2}s_{i}w_{3}^{i}$ and $g_{5}=\sum_{i=0}^{2}t_{i}w_{3}^{i}$, then $s_{2}=s_{2}'\cdot(a_{1}-b_{1})$ and $t_{2}=t_{2}'\cdot(a_{1}-c_{1})$. We can compute to get that $s_{2}'-t_{2}'=(b_{1}-c_{1})(a_{1}+b_{1}+c_{1}+w_{1})$. By Lemmas~\ref{lemma3} and \ref{lemma1}, we have $s_{2}'-t_{2}'\ne0$. Hence there is at least one of $s_{2},t_{2}$ not being 0, then there are at most $2$ solutions for $w_{3}$. If $w_{1}$ and $w_{3}$ are given, then all the remaining variables are uniquely determined. Hence there are at most $20$ different sequences of $(x,y,z,w)$ for this case.

If $a_{2}=c_{2}$, then $h_{i}=0$ for $i=6,7,8,9,10$. We can regard $h_{4}$ and $h_{5}$ as polynomials with variable $a_{1}$. Then by a MAGMA program, we have
\begin{align*}
R(h_{4},h_{5},a_{1})=(b_{1}-c_{1})^{2}(a_{3}-b_{3})^{4}(a_{3}-c_{3})^{4}(b_{3}-c_{3})^{4}.
\end{align*}

If $R(h_{4},h_{5},a_{1})\ne0$, then at least one of $h_{4},h_{5}$ is not 0. Hence there are at most 5 solutions for $w_{1}$. For any fixed $w_{1}$, through a similar discussion as above, there are at most $2$ solutions for $w_{3}$. If $w_{1}$ and $w_{3}$ are given, then all the remaining variables are uniquely determined. Hence there are at most $10$ different sequences of $(x,y,z,w)$ for this case.

If $R(h_{4},h_{5},a_{1})=0$, without loss of generality, we assume that $a_{3}=b_{3}$. Then we can compute to get that
\begin{align*}
g_{4}=(a_{3}-w_{3})^{2}(a_{1}-b_{1})\cdot g_{4}',
\end{align*}
where $g_{4}'=-w_{1}^{2}+(a_{1}+b_{1})w_{1}+a_{1}^{2}+a_{1}b_{1}+b_{1}^{2}$. It is easy to see that $a_{3}\ne w_{3}$ and $a_{1}\ne b_{1}$, then there are at most 2 solutions for $w_{1}$. For any fixed $w_{1}$, $g_{5}$ is a polynomial of $w_{3}$ with degree 2. We write $g_{5}$ as $g_{5}=\sum_{i=0}^{2}t_{i}w_{3}^{i}$, then $t_{2}=t_{2}'\cdot(a_{1}-c_{1})$. We can compute to get that $g_{4}'-t_{2}'=(b_{1}-c_{1})(a_{1}+b_{1}+c_{1}+w_{1})$. By Lemmas~\ref{lemma3} and \ref{lemma1}, we have $g_{4}'-t_{2}'\ne0$. Hence there are at most $2$ solutions for $w_{3}$. If $w_{1}$ and $w_{3}$ are given, then all the remaining variables are uniquely determined. Hence there are at most $4$ different sequences of $(x,y,z,w)$ for this case.

Therefore, if $a_{1}\ne b_{1}$, $a_{1}\ne c_{1}$, $b_{1}\ne c_{1}$ and $a_{2}=b_{2}$ or $a_{2}=c_{2}$ or $b_{2}=c_{2}$, then there are at most $20$ different sequences of $(x,y,z,w)$.

\subsection{$a_{1}\ne b_{1}$, $a_{1}\ne c_{1}$, $b_{1}\ne c_{1}$, $a_{2}\ne b_{2}$, $a_{2}\ne c_{2}$ and $b_{2}\ne c_{2}$}
For this case, we begin with the following lemmas. A theta graph $\theta_{k,t}$ is a graph made of $t$ internally disjoint paths of length $k$ connecting two endpoints.
\begin{lemma}\label{lemma2}
If $G_{p}$ contains a $\theta_{3,3}$, and the set of edges $\{da,db,dc,ax,by,cz,wx,wy,wz\}$ form a $\theta_{3,3}$, then $w_{1}=a_{1}+b_{1}+c_{1}$ or $a_{1}y_{1}-a_{1}z_{1}-b_{1}x_{1}+b_{1}z_{1}+c_{1}x_{1}-c_{1}y_{1}=0$.
\end{lemma}
\begin{proof}
We first consider the edges $da,ax,xw,wy,yb,bd$, which form a hexagon. By the definition of $G_{p}$, we have
\begin{align}
&d_{3}+a_{2}=d_{1}^{2}a_{1},&&a_{3}+d_{2}=a_{1}^{2}d_{1},\label{eq7}\\
&a_{2}+x_{3}=a_{1}x_{1}^{2},&&x_{2}+a_{3}=x_{1}a_{1}^{2},\label{eq8}\\
&x_{3}+w_{2}=x_{1}^{2}w_{1},&&w_{3}+x_{2}=w_{1}^{2}x_{1},\label{eq9}\\
&w_{2}+y_{3}=w_{1}y_{1}^{2},&&y_{2}+w_{3}=y_{1}w_{1}^{2},\label{eq10}\\
&y_{3}+b_{2}=y_{1}^{2}b_{1},&&b_{3}+y_{2}=b_{1}^{2}y_{1},\label{eq11}\\
&b_{2}+d_{3}=b_{1}d_{1}^{2},&&d_{2}+b_{3}=d_{1}b_{1}^{2}.\label{eq12}
\end{align}
Then we can compute to get that
\begin{align*}
&f_{1}:=d_{1}^{2}a_{1}-a_{1}x_{1}^{2}+x_{1}^{2}w_{1}-w_{1}y_{1}^{2}+y_{1}^{2}b_{1}-b_{1}d_{1}^{2}=0,\\
&f_{2}:=d_{1}a_{1}^{2}-a_{1}^{2}x_{1}+x_{1}w_{1}^{2}-w_{1}^{2}y_{1}+y_{1}b_{1}^{2}-b_{1}^{2}d_{1}=0,
\end{align*}
where $f_{1}$ is from the left six equations of (\ref{eq7})-(\ref{eq12}) and $f_{2}$ is from the right six equations of (\ref{eq7})-(\ref{eq12}). Regarding $f_{1},f_{2}$ as polynomials with variables $a_{1},u_{1},v_{1},b_{1},x_{1},w_{1}$, we can compute to get that
\begin{align*}
R(f_{1},f_{2},b_{1})=&(a_{1}-w_{1})(x_{1}-y_{1})(d_{1}-y_{1})(d_{1}-x_{1})(d_{1}^{2}a_{1}+d_{1}^{2}w_{1}-d_{1}a_{1}x_{1}-d_{1}a_{1}y_{1}+\\
&d_{1}x_{1}w_{1}+d_{1}w_{1}y_{1}-a_{1}x_{1}^{2}-a_{1}x_{1}y_{1}-a_{1}y_{1}^{2}+x_{1}^{2}w_{1}+x_{1}w_{1}y_{1}-w_{1}y_{1}^{2}).
\end{align*}
By Lemma~\ref{lemma1} and $f_{1}=f_{2}=0$, we have
 \begin{align*}
 &d_{1}^{2}a_{1}+d_{1}^{2}w_{1}-d_{1}a_{1}x_{1}-d_{1}a_{1}y_{1}+d_{1}x_{1}w_{1}+d_{1}w_{1}y_{1}-a_{1}x_{1}^{2}-a_{1}x_{1}y_{1}-\\
 &a_{1}y_{1}^{2}+x_{1}^{2}w_{1}+x_{1}w_{1}y_{1}-w_{1}y_{1}^{2}=0.
 \end{align*}
 Similarly, the edges $da,ax,xw,wz,zc,cd$ form a hexagon, we have
 \begin{align*}
 &d_{1}^{2}a_{1}+d_{1}^{2}w_{1}-d_{1}a_{1}x_{1}-d_{1}a_{1}z_{1}+d_{1}x_{1}w_{1}+d_{1}w_{1}z_{1}-a_{1}x_{1}^{2}-a_{1}x_{1}z_{1}-\\
 &a_{1}z_{1}^{2}+x_{1}^{2}w_{1}+x_{1}w_{1}z_{1}-w_{1}z_{1}^{2}=0.
 \end{align*}
From the above two equations, we have
\begin{align*}
(d_{1}+x_{1})(w_{1}-a_{1})y_{1}-(w_{1}+a_{1})y_{1}^{2}=(d_{1}+x_{1})(w_{1}-a_{1})z_{1}-(w_{1}+a_{1})z_{1}^{2}.
\end{align*}
Then we have
\begin{align*}
\frac{d_{1}+x_{1}}{w_{1}+a_{1}}=\frac{y_{1}+z_{1}}{w_{1}-a_{1}}.
\end{align*}
By the symmetry of $\theta_{3,3}$, we have
\begin{align*}
\frac{w_{1}+a_{1}}{d_{1}+x_{1}}=\frac{b_{1}+c_{1}}{d_{1}-x_{1}}.
\end{align*}
From the above two equations, we can get
\begin{align*}
g_{1}=(y_{1}+z_{1})(b_{1}+c_{1})-(w_{1}-a_{1})(d_{1}-x_{1})=0.
\end{align*}
By the symmetry of $\theta_{3,3}$ again, we also have
\begin{align*}
&g_{2}=(x_{1}+z_{1})(a_{1}+c_{1})-(w_{1}-b_{1})(d_{1}-y_{1})=0,\\
&g_{3}=(x_{1}+y_{1})(a_{1}+b_{1})-(w_{1}-c_{1})(d_{1}-z_{1})=0.
\end{align*}
Now we regard $g_{i}$ ($i=1,2,3$) as polynomials with variables $d_{1},w_{1}$. By a MAGMA program, we can get that
\begin{align*}
&R(g_{1},g_{2},d_{1})=h_{1}\cdot(a_{1}+b_{1}+c_{1}-w_{1}),\\
&R(g_{1},g_{3},d_{1})=h_{2}\cdot(a_{1}+b_{1}+c_{1}-w_{1}).
\end{align*}
If $a_{1}+b_{1}+c_{1}-w_{1}\ne0$, then we can compute to get that $R(h_{1},h_{2},w_{1})=(y_{1}+z_{1})(a_{1}y_{1}-a_{1}z_{1}-b_{1}x_{1}+b_{1}z_{1}+c_{1}x_{1}-c_{1}y_{1})=0$. Hence $a_{1}y_{1}-a_{1}z_{1}-b_{1}x_{1}+b_{1}z_{1}+c_{1}x_{1}-c_{1}y_{1}=0$.
\end{proof}

\begin{remark}\label{rmk1}
It is easy to see that if $d_{1}\in(\mathbb{F}_{p}^{*}\backslash S)$, then Lemma~\ref{lemma2} still holds.
\end{remark}

\begin{lemma}\label{lemma5}
If $w_{1}=a_{1}+b_{1}+c_{1}$, then there are at most $2$ different sequences of $(x,y,z,w)$.
\end{lemma}
\begin{proof}
Substituting $w_{1}=a_{1}+b_{1}+c_{1}$ into Equations (\ref{eq1})-(\ref{eq6}), we have
\begin{align*}
&f_{1}:=a_{2}-w_{2}+x_{1}^{2}(b_{1}+c_{1})=0,\\
&f_{2}:=a_{3}-w_{3}-x_{1}(a_{1}^{2}-(a_{1}+b_{1}+c_{1})^{2})=0,\\
&f_{3}:=b_{2}-w_{2}+y_{1}^{2}(a_{1}+c_{1})=0,\\
&f_{4}:=b_{3}-w_{3}-y_{1}(b_{1}^{2}-(a_{1}+b_{1}+c_{1})^{2})=0,\\
&f_{5}:=c_{2}-w_{2}+z_{1}^{2}(a_{1}+b_{1})=0,\\
&f_{6}:=c_{3}-w_{3}-z_{1}(c_{1}^{2}-(a_{1}+b_{1}+c_{1})^{2})=0.
\end{align*}
Now we regard $f_{i}$ ($i=1,2,\dots,6$) as polynomials with variables $x_{1},y_{1},z_{1},w_{2},w_{3}$. By a MAGMA program, we can get that
\begin{align*}
&R(f_{1},f_{2},x_{1})=g_{1}\cdot(b_{1}+c_{1}),\\
&R(f_{3},f_{4},y_{1})=g_{2}\cdot(a_{1}+c_{1}),\\
&R(f_{5},f_{6},z_{1})=g_{3}\cdot(a_{1}+b_{1}),\\
&R(g_{1},g_{2},w_{2})=g_{4},\\
&R(g_{1},g_{3},w_{2})=g_{5},
\end{align*}
where $g_{4}$ and $g_{5}$ are polynomials of $w_{3}$ with degree $2$. We write $g_{4}$ and $g_{5}$ as $g_{4}=\sum_{i=0}^{2}s_{i}w_{3}^{i}$ and $g_{5}=\sum_{i=0}^{2}t_{i}w_{3}^{i}$, then $s_{2}=s_{2}'\cdot(a_{1}-b_{1})$ and $t_{2}=t_{2}'\cdot(a_{1}-c_{1})$. Now we regard $s_{2}',t_{2}'$ as polynomials with variable $c_{1}$, then $R(s_{2}',t_{2}',c_{1})=(a_{1}-b_{1})(a_{1}+b_{1})(a_{1}^{2}+a_{1}b_{1}+b_{1}^{2})\ne0$. Hence there is at least one of $s_{2},t_{2}$ not being 0, then there are at most $2$ solutions for $w_{3}$. If $w_{1}$ and $w_{3}$ are given, then all the remaining variables are uniquely determined. Hence there are at most $2$ different sequences of $(x,y,z,w)$ for this case.
\end{proof}

By Lemma~\ref{lemma1}, if $d,a,b,c,x,y,z,w$ form a $\theta_{3,3}$ with edge set $\{da,db,dc,ax,by,cz,wx,wy,wz\}$, then $a_{3}\ne b_{3}$.
\begin{lemma}\label{lemma6}
If $a_{1}y_{1}-a_{1}z_{1}-b_{1}x_{1}+b_{1}z_{1}+c_{1}x_{1}-c_{1}y_{1}=0$ and $a_{3}\ne b_{3}$, then there are at most $27$ different sequences of $(x,y,z,w)$.
\end{lemma}
\begin{proof}
Let $f_{7}=a_{1}y_{1}-a_{1}z_{1}-b_{1}x_{1}+b_{1}z_{1}+c_{1}x_{1}-c_{1}y_{1}=0$, then we regard $f_{7}$ and $f_{i}$ ($i=1,2,\dots,6$) in Equations (\ref{eq1})-(\ref{eq6}) as polynomials with variables $x_{1},y_{1},z_{1},w_{1},w_{2},w_{3}$. Let $g_{1}=f_{1}-f_{3}$, $g_{2}=f_{1}-f_{5}$ and $g_{3}=f_{2}-f_{4}$. By a MAGMA program, we can get that
\begin{align*}
&R(g_{1},g_{2},w_{1})=g_{5},\\
&R(g_{1},g_{3},w_{1})=g_{6}\cdot(x_{1}-y_{1}),\\
&R(f_{7},g_{5},x_{1})=h_{1}\cdot(y_{1}-z_{1}),\\
&R(f_{7},g_{6},x_{1})=h_{2},\\
&R(h_{1},h_{2},z)=(b_{1}-c_{1})^{3}(a_{1}-b_{1})^{6}(a_{1}y_{1}^{2}-b_{1}y_{1}^{2}-a_{2}+b_{2})^{2}\cdot s,
\end{align*}
where $s$ is a polynomial of $y_{1}$ with degree $8$. Since at most one of $u,-u$ belongs to $S$, then $a_{1}y_{1}^{2}-b_{1}y_{1}^{2}-a_{2}+b_{2}=0$ has at most 1 solution for $y_{1}$. We write $s=\sum_{i=0}^{8}s_{i}y_{1}^{i}$, then we can compute to get that $s_{8}=s_{8}'\cdot(b_{1}-c_{1})^{3}(a_{1}-c_{1})^{4}(a_{1}-b_{1})^{4}$ and $s_{7}=s_{7}'\cdot(b_{1}-c_{1})^{3}(a_{3}-b_{3})(a_{1}-c_{1})^{3}(a_{1}-b_{1})^{3}$. We regard $s_{8}'$ and $s_{7}'$ as polynomials of $a_{1}$, then $R(s_{7}',s_{8}',a_{1})=(b_{1}+c_{1})(b_{1}+5c_{1})(b_{1}^{2}+b_{1}c_{1}+c_{1}^{2})\ne0$. Hence there are at most $9\ (=8+1)$ solutions for $y_{1}$ (note that we have proved $a_{1}y_{1}^{2}-b_{1}y_{1}^{2}-a_{2}+b_{2}=0$ has at most 1 solution for $y_{1}$ previously). For any given $y_{1}$, $h_{1}$ is a polynomial of $z_{1}$ with degree 3. We write $h_{1}$ as $h_{1}=\sum_{i=0}^{3}s_{i}z_{1}^{i}$, then $s_{3}=(a_{1}-c_{1})(a_{1}-b_{1})^{2}\ne0$. Hence, there are at most 3 solutions for $z_{1}$. If $y_{1}$ and $z_{1}$ are given, then all the remaining variables are uniquely determined. Hence there are at most $27$ different sequences of $(x,y,z,w)$ for this case.
\end{proof}

Now we regard $f_{i}$ ($i=1,2,\dots,6$) in Equations (\ref{eq1})-(\ref{eq6}) as polynomials with variables $x_{1},y_{1},z_{1},w_{1},w_{2},w_{3}$. By a MAGMA program, we can get that
\begin{align*}
&R(f_{1},f_{2},x_{1})=g_{1}\cdot(a_{1}-w_{1}),\\
&R(f_{3},f_{4},y_{1})=g_{2}\cdot(b_{1}-w_{1}),\\
&R(f_{5},f_{6},z_{1})=g_{3}\cdot(c_{1}-w_{1}).
\end{align*}
By Lemma~\ref{lemma1}, we have $a_{1}\ne w_{1}$ and $c_{1}\ne w_{1}$. Then we can compute to get that
\begin{align*}
&R(g_{1},g_{2},w_{2})=g_{4},\\
&R(g_{1},g_{3},w_{2})=g_{5},\\
&R(g_{4},g_{5},w_{3})=h\cdot (a_{1}-w_{1})^{2}(a_{1}+w_{1})^{4},
\end{align*}
where $h$ is a polynomial of $w_{1}$ with degree $10$. By Lemma~\ref{lemma3}, $a_{1}+w_{1}\ne0$. We can write $h$ as $h=\sum_{i=0}^{10}h_{i}w_{1}^{i}$.

 If at least one of $h_{i}$ is not 0, then there are at most 10 solutions for $w_{1}$. For any fixed $w_{1}$, $g_{4}$ and $g_{5}$ are polynomials of $w_{3}$ with degree 2. We write $g_{4}$ and $g_{5}$ as $g_{4}=\sum_{i=0}^{2}s_{i}w_{3}^{i}$ and $g_{5}=\sum_{i=0}^{2}t_{i}w_{3}^{i}$, then $s_{2}=s_{2}'\cdot(a_{1}-b_{1})$ and $t_{2}=t_{2}'\cdot(a_{1}-c_{1})$. We can compute to get that $s_{2}'-t_{2}'=(b_{1}-c_{1})(a_{1}+b_{1}+c_{1}+w_{1})$. By Lemmas~\ref{lemma3} and \ref{lemma1}, we have $s_{2}'-t_{2}'\ne0$. Hence there is at least one of $s_{2},t_{2}$ not being 0, then there are at most $2$ solutions for $w_{3}$. If $w_{1}$ and $w_{3}$ are given, then all the remaining variables are uniquely determined. Hence there are at most $20$ different sequences of $(x,y,z,w)$ for this case.

If $h_{i}=0$ for $0\le i\le10$. Now we regard $h_{i}$ as polynomials with variables $b_{2},c_{2},c_{3}$.  We can compute to get that
\begin{align*}
h_{10}=(a_{1}b_{2}-a_{1}c_{2}-a_{2}b_{1}+a_{2}c_{1}+b_{1}c_{2}-b_{2}c_{1})^2.
\end{align*}
Let
$h_{10}'=a_{1}b_{2}-a_{1}c_{2}-a_{2}b_{1}+a_{2}c_{1}+b_{1}c_{2}-b_{2}c_{1}.$ By a MAGMA program, we can get that
\begin{align*}
&R(h_{10}',h_{8},c_{2})=(a_{2}-b_{2})(a_{1}-b_{1})\cdot k_{1},\\
&R(h_{10}',h_{6},c_{2})=(a_{2}-b_{2})(a_{1}-b_{1})\cdot k_{2},\\
&R(k_{1},k_{2},c_{3})=(b_{1}-c_{1})^{4}(a_{1}-c_{1})^{4}(a_{1}-b_{1})^{4}\cdot r_{1}^{2},\\
&R(k_{1},k_{2},b_{2})=(b_{1}-c_{1})^{2}(a_{1}-c_{1})^{2}(a_{1}-b_{1})\cdot r_{2}^{2},
\end{align*}
where
\begin{align}
&r_{1}=a_{1}^{3}a_{2}-a_{1}^{3}b_{2}+a_{1}^{2}a_{2}b_{1}-a_{1}^{2}b_{1}b_{2}-a_{1}a_{2}b_{1}^{2}+a_{1}b_{1}^{2}b_{2}-a_{2}b_{1}^{3}-a_{3}^{2}+2a_{3}b_{3}+b_{1}^{3}b_{2}-b_{3}^{2},\label{eq13}\\
&r_{2}=a_{1}^{2}b_{3}-a_{1}^{2}c_{3}-a_{3}b_{1}^{2}+a_{3}c_{1}^{2}+b_{1}^{2}c_{3}-b_{3}c_{1}^{2}.\label{eq14}
\end{align}

We can also compute to get that
\begin{align*}
&R(h_{10}',h_{8},b_{2})=(a_{2}-c_{2})(a_{1}-c_{1})\cdot k_{3},\\
&R(h_{10}',h_{6},b_{2})=(a_{2}-c_{2})(a_{1}-c_{1})\cdot k_{4},\\
&R(k_{3},k_{4},c_{3})=(b_{1}-c_{1})^{4}(a_{1}-c_{1})^{2}(a_{1}-b_{1})^{4}\cdot r_{3}^{2},
\end{align*}
where
\begin{align}
r_{3}=&a_{1}^{4}a_{2}-a_{1}^{4}c_{2}-2a_{1}^{2}a_{2}b_{1}^{2}+2a_{1}^{2}b_{1}^{2}c_{2}-a_{1}a_{3}^{2}+2a_{1}a_{3}b_{3}-a_{1}b_{3}^{2}+a_{2}b_{1}^{4}+a_{3}^{2}c_{1}-\notag\\
&2a_{3}b_{3}c_{1}-b_{1}^{4}c_{2}+b_{3}^{2}c_{1}.\label{eq15}
\end{align}

Now we define $d_{1}:=\frac{a_{3}-b_{3}}{a_{1}^{2}-b_{1}^{2}}$, $d_{2}:=a_{1}^{2}(\frac{a_{3}-b_{3}}{a_{1}^{2}-b_{1}^{2}})-a_{3}$, and $d_{3}:=a_{1}(\frac{a_{3}-b_{3}}{a_{1}^{2}-b_{1}^{2}})^{2}-a_{2}$. Then by $r_{1}=r_{2}=r_{3}=0$ (see Equations (\ref{eq13}), (\ref{eq14}) and (\ref{eq15})), it is easy to check that
\begin{align*}
&a_{2}+d_{3}=a_{1}d_{1}^{2}, && a_{3}+d_{2}=a_{1}^{2}d_{1},\\
&b_{2}+d_{3}=b_{1}d_{1}^{2}, && b_{3}+d_{2}=b_{1}^{2}d_{1},\\
&c_{2}+d_{3}=c_{1}d_{1}^{2}, && c_{3}+d_{2}=c_{1}^{2}d_{1}.
\end{align*}

Then the vertex $d=(d_{1},d_{2},d_{3})$ is a common neighbor of $a,b,c$. Hence the vertex $d,a,b,c,x,\\y,z,w$ form a $\theta_{3,3}$, by Lemmas~\ref{lemma2}, \ref{lemma5} and \ref{lemma6}, there are at most $29$ different sequences of $(x,y,z,w)$ for this case.

\begin{remark}
Note that $d_{1}$ may not be in the set $S$, and then $d\not\in V$, but by Remark~\ref{rmk1}, we still have the same result. Actually, we do not need the notation $\theta_{3,3}$. If the point $d=(d_{1},d_{2},d_{3})$ satisfies the above equations with $a,b,c$, then we have the statements of Lemma~\ref{lemma2}.
\end{remark}

Therefore, if $a_{1}\ne b_{1}$, $a_{1}\ne c_{1}$, $b_{1}\ne c_{1}$, $a_{2}\ne b_{2}$, $a_{2}\ne c_{2}$ and $b_{2}\ne c_{2}$, then there are at most $29$ different sequences of $(x,y,z,w)$.
\subsection{Proof of Theorem~\ref{thm}}
From the previous discussions, for any given $a,b,c\in V$, there are at most $29$ different sequences of $(x,y,z,w)$ such that $ax,xw,by,yw,cz,zw$ are edges in $G_{p}$. Hence $G_{p}$ is $K_{3,30}'$-free.

\section{Conclusion and remarks}\label{conclusion}
In this paper, we study the Tur\'{a}n number of 1-subdivision of $K_{3,t}$. More precisely, we provide an explicit construction giving
\begin{align*}
\text{ex}(n,K_{3,30}')=\Omega(n^{\frac{4}{3}}),
\end{align*}
which makes progress on the known estimation for the smallest value of $t$ concerning a problem posed by Conlon, Janzer and Lee \cite{CJL19}. It would be interesting to consider the Conlon-Janzer-Lee problem for $s\ge4$.

\end{document}